\documentclass{amsart}
\usepackage{amsmath}
\usepackage{amssymb}
\usepackage{amsthm}
\usepackage{tikz}
\usepackage{hyperref}
\hypersetup{colorlinks=false}

\theoremstyle{plain}

\newtheorem{theorem}{Theorem}[section]
\newtheorem{lemma}[theorem]{Lemma}
\newtheorem{proposition}[theorem]{Proposition}
\newtheorem{corollary}[theorem]{Corollary}

\theoremstyle{definition}
\newtheorem{definition}[theorem]{Definition}
\newtheorem{remark}[theorem]{Remark}

\numberwithin{equation}{section}
\newcommand\fantome[1]{}

\DeclareMathOperator{\ev}{ev}

\DeclareMathOperator{\sgn}{sgn}
\DeclareMathOperator{\Drin}{Drin}
\DeclareMathOperator{\Sht}{Sht}

\begin{document}

\title{Special functions and twisted $L$-series}

\author{Bruno Angl\`es  \and Tuan Ngo Dac \and Floric Tavares Ribeiro}

\address{
Normandie Universit\'e
Universit\'e de Caen Normandie,
Laboratoire de Math\'ematiques Nicolas Oresme,
CNRS UMR 6139,
Campus II, Boulevard Mar\'echal Juin,
B.P. 5186,
14032 Caen Cedex, France.
}
\email{bruno.angles@unicaen.fr,tuan.ngodac@unicaen.fr,floric.tavares-ribeiro@unicaen.fr}


\begin{abstract} We introduce a generalization of the Anderson-Thakur special function, and we prove a rationality result for several variable  twisted $L$-series associated to shtuka functions.
\end{abstract}

\date{  \today}

\maketitle

\footnotetext[1]{The second author was partially supported by ANR Grant PerCoLaTor ANR-14-CE25-0002.}

\tableofcontents

\section{Introduction}

Let $X=\mathbb P^1/\mathbb F_q$ be the projective line over a finite field $\mathbb F_q$ having $q$ elements and let $K$ be its function field. Let $\infty$ be a closed point of $X$ of degree $d_\infty=1.$  Then $K=\mathbb  F_q(\theta)$ for some $\theta \in K$ such that $\theta$ has a pole of order one at $\infty.$ We set $A=\mathbb  F_q[\theta].$ Following Anderson (\cite{AND}, see also \cite{THA2}), we consider:
$$Y= K\otimes_{\mathbb F_q}X.$$
Let $\mathbb K={\rm Frac}(K\otimes_{\mathbb F_q}K)$ be the function field of $Y.$ We identify $K$ with $K\otimes 1\subset \mathbb K.$ If we set $t=1\otimes \theta,$ then $\mathbb  K=K(t).$  Let $\tau : \mathbb K\rightarrow \mathbb K$ be the homomorphism of $\mathbb F_q(t)$-algebras such that:
$$\forall x\in K, \quad \tau(x)=x^q.$$
Let $\bar{\infty}\in Y(K)$ be the pole of $t,$ and let $\xi \in Y(K)$ be the point corresponding to the kernel of the homomorphism of $K$-algebras $K\otimes_{\mathbb F_q}K\rightarrow K$ which sends $t$ to $\theta.$ Then the divisor of $f:=t-\theta$ is equal to $(\xi)-(\bar{\infty}).$ The function $t-\theta$ is  a shtuka function, and in particular:
$$\forall a\in A, \quad a(t)=\sum_{k=0}^{\deg_\theta a} C_{a,i} f\cdots f^{(i-1)}, \text{ with } C_{a,i} \in A.$$
The map $C: A\rightarrow A\{\tau\}, a\mapsto C_a:= \sum_{k=0}^{\deg_\theta a} C_{a,i}\tau ^i$  is a homomorphism of $\mathbb F_q$-algebras called the Carlitz module. Note that:
$$C_\theta= \theta+\tau.$$
There exists a unique element $\exp_C \in K\{\{\tau\}\}$ such that $\exp_C\equiv 1\pmod{\tau} $ and:
$$ \forall a\in A, \quad \exp_C a= C_a \exp_C.$$

Let $\mathbb C_\infty$ be the completion of a fixed algebraic closure of $K_\infty :=\mathbb F_q((\frac{1}{\theta})).$ Then $\exp_C$ defines an entire function on $\mathbb C_\infty,$ and:
$${\rm Ker }\exp_C =\widetilde{\pi} A,$$
for some $\widetilde{\pi} \in \mathbb C_\infty^\times$ (well-defined modulo $\mathbb F_q^\times$) called the Carlitz period. We consider $\mathbb T$ the Tate algebra in the variable $t$ with coefficients in $\mathbb C_\infty,$ i.e. $\mathbb T:=\mathbb C_\infty\widehat{\otimes}_{\mathbb F_q} A.$ Let $\tau: \mathbb T\rightarrow \mathbb T$ be the continuous homomorphism of $\mathbb F_q[t]$-algebras such that $\forall x\in \mathbb C_\infty, \tau(x)=x^q.$ Anderson and Thakur (\cite{AND&THA}) showed that:
$$\{ x\in \mathbb T, \tau(x)=fx\} =\omega \mathbb F_q[t],$$
where $\omega \in \mathbb T^\times$ is such that:
$$f\omega\mid_{\xi}=\widetilde{\pi}.$$
The function $\omega$ is called the Anderson-Thakur special function attached to the Carlitz module $C.$ This function is intimately connected to Thakur-Gauss sums  (\cite{ANG&PEL}).

\medskip

In 2012, Pellarin (\cite{PEL})  initiated  the study of a twist of the Carlitz module by the shtuka function $f.$ Let's consider the following homomorphism of $\mathbb F_q$-algebras $\varphi: A\rightarrow A[t]\{\tau\}, \theta\mapsto \theta+f\tau.$ Then, one observes that $C$ and $\varphi$ are isomorphic over $\mathbb T,$ i.e. we have the following equality in $\mathbb T\{\tau\}:$
$$\forall a\in A, \quad C_a\omega = \omega \varphi_a.$$
To such an object, one can associate the special value of some twisted $L$-function (see \cite{APT}):
$$\mathcal L= \sum_{a\in A,a\,  {\rm monic}} \frac{a(t)}{a} \in \mathbb T^\times.$$
Then, using the Anderson log-algebraicity Theorem for the Carlitz module (\cite{AND2}, see also \cite{PAP}, \cite{APTR}), Pellarin proved the following remarkable rationality result:
$$\frac{\mathcal L \omega}{\widetilde{\pi}}= \frac{1}{f}\in \mathbb K.$$
This result has been extended to the case of ``several variables'' (\cite{APT}, \cite{DEM}) using methods developed by Taelman (\cite{TAE1}, \cite{TAE2}, \cite{AT},  \cite{FAN1}, \cite{FAN2}, \cite{FAN0}). This kind of rationality results leads to new advances in the arithmetic of function fields (see \cite{APT}, \cite{ATR}, \cite{ANDTR3}).

\medskip

The aim of this paper is to extend the previous results to the general context, i.e. for any smooth projective geometrically irreducible curve $X/\mathbb F_q$ of genus $g$ and any closed point $\infty$ of degree $d_\infty$ of $X.$ In particular, we obtain a rationality result similar to that of Pellarin (Theorem \ref{TheoremS3-1}). Our result involves twisted $L$-series (see \cite{ANDTR2}) and a generalization of the Anderson-Thakur special function. The involved techniques are based on ideas developed in \cite{ANDTR3} where an analogue of Stark Conjectures is proved for sign-normalized rank one Drinfeld modules.

We should mention that Green and Papanikolas (\cite{GRE&PAP}) have recently studied the particular case $g=1$ and $d_\infty=1$ and, in this case, they have obtained explicit formulas similar to that obtained by Pellarin (in the case $g=0$ and $d_\infty=1$).

\medskip

The first  author would like to thank David Goss and Federico Pellarin for interesting discussions around the topics considered in  this article. The authors  dedicate this work to David Goss.


\section{Notation and background}

\subsection{Notation}${}$\par

Let $X/\mathbb F_q$ be a smooth projective geometrically irreducible curve of genus $g$, and $\infty$ be a closed point of degree $d_\infty$ of $X.$ Denote by $K$ the function field of $X,$ and by $A$ the ring of elements of $K$ which are regular outside $\infty.$ The  completion $K_\infty$ of $K$ at the place $\infty$ has residue field $\mathbb F_\infty.$ We fix an algebraic closure $\overline{K}_\infty$ of $K_\infty$ and denote by $\mathbb C_\infty$ the completion of $\overline{K}_\infty$. 

We will fix a sign function $\sgn: K_\infty^\times \rightarrow \mathbb F_\infty^\times$ which is a group homomorphism  such that $\sgn\mid_{\mathbb F_\infty^\times}={\rm Id}\mid_{\mathbb F_\infty^\times}.$ We fix  $\pi \in K\cap {\rm Ker}(\sgn)$ and  such that $K_\infty =\mathbb F_\infty((\pi)).$ Let $v_\infty:\mathbb C_\infty \rightarrow \mathbb  Q\cup\{+\infty\}$ be the valuation on $\mathbb C_\infty$ normalized such that $v_\infty(\pi)=1.$  Observe that:
$$\forall x\in K^\times, \quad \deg (xA)= -d_\infty v_\infty(x).$$
Let $\overline{K}$ be the algebraic closure of $K$ in $\mathbb C_\infty.$

\medskip

Let $\mathcal I(A)$ be the group of non-zero fractional ideals of $A.$ We have a natural surjective group homomorphism $\deg :\mathcal I(A)\rightarrow \mathbb Z,$ such that for $I \in \mathcal I(A), I\subset A$, we have:
$$\deg I={\rm dim}_{\mathbb F_q} A/I.$$
Let $\mathcal P(A)= \{ xA, x\in K^\times\},$ then ${\rm Pic}(A)= \frac{\mathcal I(A)}{\mathcal P(A)}$ is a finite abelian group.

\medskip

Let $I_K$ be the group of id\`eles of $K,$ and $H/K$ be the finite abelian extension of $K,$ $H\subset \mathbb C_\infty,$  corresponding via class field theory to the following subgroup of $I_K:$
$$K^\times \, \ker \sgn \, \prod_{v \neq \infty} O_v^\times ,$$
where for a place $v \neq \infty$ of $K,$ $O_v^\times$ denotes the group of units of the $v$-adic completion of $K.$ Then $H/K$ is a finite extension of degree $\mid {\rm Pic}(A)\mid\frac{q^{d_\infty}-1}{q-1},$ unramified outside $\infty,$ and the decomposition group of $\infty$ in $H/K$ is equal to its inertia group and is isomorphic to $\frac{\mathbb F_\infty^{\times}}{\mathbb F_q^{\times}}.$  Set $G={\rm Gal}(H/K).$ If we define $\mathcal P_+(A)=\{ xA, x\in K^\times, \sgn(x)=1\},$ then the Artin map 
$$(\cdot,H/K):\mathcal I(A) \longrightarrow G.$$
induces a group isomorphism:
$$\frac{\mathcal I(A)}{\mathcal P_+(A)}\simeq G.$$
For $I \in \mathcal I(A),$ we set:
$$\sigma_I=(I,H/K) \in G.$$

Let $H_A$ be the Hilbert class field of $A,$ i.e. $H_A/K$ corresponds to the following subgroup of the id\`eles of $K:$
$$K^\times K_\infty^\times\prod_{v\not =\infty} O_v^\times.$$
Then $H/H_A$ is totally ramified at the places of $H_A$ above $\infty.$ Furthermore:
 $${\rm Gal}(H/H_A)\simeq \frac{\mathbb F_\infty^\times}{\mathbb F_q^\times}.$$
We denote by $B$ the integral closure of $A$ in $H$ and $B'$ the integral closure of $A$ in $H_A.$ Observe that $\mathbb F_\infty \subset B.$

\subsection{Sign-normalized rank one Drinfeld modules}${}$\par

We define the map $\tau: \mathbb C_\infty \rightarrow \mathbb C_\infty, x\mapsto x^q.$  By definition, a sign-normalized rank one Drinfeld module is a homomorphism of $\mathbb F_q$-algebras $\phi:A\rightarrow \mathbb C_\infty \{\tau\}$ such that there exists $n(\phi) \in \{0, \cdots , d_\infty -1\}$ with the following property:
$$\forall a\in A, \quad \phi_a= a+\cdots +\sgn(a)^{q^{n(\phi)}} \tau^{\deg a}.$$

Let $n \in \{0, \cdots , d_\infty -1\}.$ We denote by $\Drin_n$ the set of sign-normalized rank one Drinfeld modules $\phi$ with $n(\phi)=n,$ and by $\Drin =\cup_{n=0}^{d_\infty-1} \Drin_n$ the set of sign-normalized rank one Drinfeld modules. By \cite{GOS}, Corollary 7.2.17, $\Drin$ is a finite set and we have:
$$\mid \Drin\mid=\mid {\rm Pic}(A)\mid\frac{q^{d_\infty}-1}{q-1}.$$

Let $\phi \in \Drin$ be a sign-normalized rank one Drinfeld module, we say that $\phi$ is standard if ${\rm Ker}\exp_\phi$ is a free $A$-module, where $\exp_\phi:\mathbb C_\infty \rightarrow \mathbb C_\infty$ is the exponential map attached to $\phi$ (see for example \cite{GOS}, paragraph 4.6).

\begin{lemma} \label{LemmaS1-1} 
Let $n\in \{0, \cdots, d_\infty-1\}.$ We have:
$$\mid \Drin_n\mid= \frac{1}{d_\infty}\mid {\rm Pic}(A)\mid\frac{q^{d_\infty}-1}{q-1}.$$
Let $\phi$ in $\Drin_n$ and let $[\phi]$ denote the set of the $\phi' $ in $\Drin_n$ which are isomorphic to $\phi.$ Then:
$$\forall \phi \in \Drin_n, \quad \mid [\phi]\mid = \frac{q^{d_\infty}-1}{q-1}.$$
In particular, if $[\Drin_n]=\{Ê[\phi], \phi \in \Drin_n\},$ we have:
$$\mid [\Drin_n]\mid= \frac{1}{d_\infty}\mid {\rm Pic}(A)\mid.$$
\end{lemma}

\begin{proof} 

Let $\psi:A\rightarrow H\{\tau\}$ be a sign-normalized rank one Drinfeld module (see \cite{GOS}, chapter 7).  Let $n(\psi)\in \mathbb Z$ be such that:
$$\forall a\in A, \quad \psi_a= a+\cdots + \sgn(a)^{q^{n(\psi)}}\tau^{\deg a}.$$ Then the set of sign-normalized rank one Drinfeld modules is exactly $\Drin=\{ \psi^{\sigma}, \sigma \in G\}.$ Let $\sigma \in G$ and write $\sigma =(I, H/K)$ for some $I \in \mathcal I(A).$  We have:
$$\forall a\in A, \quad \psi^{\sigma}_a= a+\cdots + \sgn(a)^{q^{n(\psi)+\deg (I)}}\tau^{\deg a}.$$
Note that $\deg :\mathcal I(A)\rightarrow \mathbb Z$ induces a surjective homomorphism of finite abelian groups:
$$\deg : \frac{\mathcal I(A)}{\mathcal P_+(A)} \rightarrow \frac{\mathbb Z}{d_\infty\mathbb Z}.$$
Since there are exactly $\mid {\rm Pic}(A)\mid\frac{q^{d_\infty}-1}{q-1}$ sign-normalized rank one Drinfeld modules and $d_\infty$ divides $\mid {\rm Pic}(A)\mid$, we get the  first assertion.

\medskip

Let $\phi \in \Drin_n$ and let $\phi' \in [\phi].$ Then there exists $\alpha \in \mathbb C_\infty^\times$ such that:
$$\forall a\in A, \quad \alpha \phi_a= \phi'_a\alpha.$$
Thus, $\alpha \in \mathbb F_\infty^\times.$  Since ${\rm End}_{\mathbb C_\infty}(\phi)=\{ \phi_a, a\in A\},$ we obtain:
$${\rm End}_{\mathbb C_\infty}(\phi)\cap \mathbb F_\infty=\mathbb F_q.$$
Hence,
$$\mid[\phi]\mid= \frac{q^{d_\infty}-1}{q-1}.$$
\end{proof}

\begin{lemma}\label{LemmaS1-2}
There are exactly $\frac{q^{d_\infty}-1}{q-1}$ standard elements in $\Drin$. Furthermore, if $\phi$ is such a Drinfeld module, then  $[\phi]$ is the set of standard elements in $\Drin$.
\end{lemma}
\begin{proof} 
By \cite{GOS}, Corollary 4.9.5 and \cite{GOS}, Theorem 7.4.8, there exists $\phi\in \Drin$ such that $\phi$ is standard. In particular, $\Drin=\{\phi^\sigma, \sigma \in G\}.$
Again, by \cite{GOS}, Corollary 4.9.5 and \cite{GOS}, Theorem 7.4.8, the Drinfeld module $\phi^{\sigma}$ is standard if and only if  $\sigma\mid_{H_A}={\rm Id}_{H_A}.$ The Lemma follows.
\end{proof}

\subsection{Shtuka functions} ${}$\par

Let $\bar X= \mathbb C_\infty\otimes_{\mathbb F_q}X,$ $\bar A=\mathbb C_\infty\otimes_{\mathbb F_q}A,$ and let $F$ be the function field of $\bar X,$ i.e. $F={\rm Frac}(\bar A).$ We will identify $\mathbb C_\infty$ with its image $\mathbb C_\infty\otimes 1 $ in $F.$ There are $d_\infty$ points  in $\bar X(\mathbb C_\infty)$ above $\infty,$ and we denote the set of such points by $S_\infty.$ Observe that $\bar A$ is the set of elements of $F/\mathbb C_\infty$ which are ``regular outside $\infty$''. We denote by $\tau: F\rightarrow  F$ the homomorphism of $K$-algebras such that:
$$\tau\mid_{\bar A}= \tau\otimes 1.$$
For $m\in \mathbb Z,$ we also set:
$$\forall x\in F, \quad x^{(m)}=\tau^m (x).$$
Let $P$ be a point of $\bar X (\mathbb C_\infty).$ We denote by $P^{(i)}$ the point of $\bar X(\overline{K})$ obtained by applying $\tau^i$ to  the coordinates of $P.$  If $D=\sum_{j=1}^n   n_{P_j} P_j \in {\rm Div}(\bar X),$ with $P_j\in \bar X(\mathbb C_\infty),$  and $n_{P_j}\in \mathbb Z,$ we set:
$$D^{(i)}=\sum_{j=1}^n  n_{P_j} P_j^{(i)}.$$
If $D=(x)$, $x\in F^\times,$ then:
$$D^{(i)}= (x^{(i)}).$$
We consider $\xi\in \bar X(\mathbb C_\infty)$ the point corresponding to the kernel of the map:
$$\bar A\rightarrow \mathbb C_\infty , \quad \sum_i x_i\otimes a_i\mapsto \sum x_ia_i.$$
Let $\rho: K\rightarrow F, x\mapsto 1\otimes x$ and set $t=\rho(\pi^{-1}).$

\medskip

Let $\bar {\infty} \in S_\infty.$ We identify the  $\bar \infty$-adic completion of $F$ to $$\mathbb C_\infty((\frac{1}{t})).$$
Let $\sgn_{\bar{\infty}}:\mathbb C_\infty((\frac{1}{t}))^\times\rightarrow \mathbb C_\infty^\times$ be the group homomorphism such that ${\rm Ker}( \sgn_{\bar{\infty}})= t^{\mathbb Z} \times (1+\frac{1}{t}\mathbb C_\infty[[\frac{1}{t}]]),$ and $ \sgn_{\bar {\infty}} \mid _{\mathbb C_\infty^\times}= {\rm Id}\mid_{\mathbb C_\infty^{\times}}.$ 


\medskip

Let $\phi \in \Drin$.  For  $a\in A,$ we write $\phi_a=\sum_{i=0}^{\deg a} \phi_{a,i} \tau^i,$ $\phi_{a,i}\in H.$ By \cite{GOS}, chapter 6, and \cite{GOS}, Proposition 7.11.4, there exists $\bar{\infty}\in S_\infty$ and $f_{\phi}\in F^\times$  such that:
$$\forall a\in A, \quad \rho(a)=\sum_{i=0}^{\deg a} \phi_{a,i} f_{\phi}\cdots f_{\phi}^{(i-1)},$$
and  the divisor of $f_{\phi}$ is of the form:
$$(f_{\phi})=V^{(1)}-V+(\xi)-(\bar \infty),$$
where $V$ is some effective divisor of degree $g.$ Let $(\infty)=\sum_{\bar {\infty}'\in S_\infty} (\bar {\infty}').$ Set 
$$W(\mathbb C_\infty) =\cup_{m\geq 0} L(V+m (\infty)),$$ and $$L(V+m(\infty))=\{ x\in F^\times, (x)+V+m(\infty)\geq 0\} \cup\{0\}.$$  We have:
$$W(\mathbb C_\infty)=\oplus _{i\geq 0} \mathbb C_\infty f_{\phi}\cdots f_{\phi}^{(i-1)}.$$

The function $f_{\phi}$ is called the shtuka function attached to $\phi,$ and we say that $\phi$ is the signed-normalized rank one Drinfeld module associated to $f_{\phi}.$
We define the set of shtuka functions to be:
$$\Sht=\{ f_{\phi}, \phi \in \Drin\}.$$
Then, the map $\Drin\rightarrow \Sht, \phi\rightarrow f_{\phi}$ is a bijection called the Drinfeld correspondence. 

\begin{remark}
There is a misprint in \cite{GOS}, page 229. In fact, as we will see in the proof of of Lemma \ref{LemmaS2-1.1}, when $d_\infty>1,$ we do not have: $\sgn_{\bar {\infty}^{(-1)}} (f_{\phi})^{\frac{q^{d_\infty}-1}{q-1}}= 1$ as stated in the \textit{loc. cit.}
\end{remark}


\section{Special functions attached to shtuka functions}

\subsection{Basic properties of a shtuka function}
${}$\par

Let $\mathbb H={\rm Frac} (H\otimes_{\mathbb F_q}A),$ and  $\mathbb K={\rm Frac} (K\otimes_{\mathbb F_q}A).$ Recall that $G={\rm Gal}(H/K)$ and we will identify $G$ with the Galois group of $\mathbb H/\mathbb K.$ Let $f\in \Sht,$ and let  $\phi\in \Drin_{n(\phi)}$ be the sign-normalized rank one Drinfeld module attached to $f$ for some $n(\phi)\in \{0,\ldots, d_\infty-1\}.$ Then $\phi:A\rightarrow B\{\tau\}$ is a homomorphism of $\mathbb F_q$-algebras such that:
$$\forall a \in A , \quad \phi_a=\sum_{i=0}^{\deg a} \phi_{a,i} \tau^i,$$
where $\phi_{a,0}=a,$ $\phi_{a, \deg a}=\sgn (a)^{q^{n(\phi)}}, $ and $\rho(a)= \sum_{i=0}^{\deg a} \phi_{a,i} f\cdots f^{(i-1)}.$ Recall that there exists an effective $\mathbb H$-divisor $V$ (\cite{GOS}, chapter 6) of degree $g$ such that the divisor of $f$  is:
$$(f)=V^{(1)}-V+(\xi)-(\bar \infty),$$
for some $\bar{ \infty} \in S_\infty.$
By \cite{GOS}, Lemma 7.11.3, $\xi, \bar \infty^{(-1)}$ do not belong to the support of $V.$
Let  $v_{\bar \infty}$ be the normalized valuation on $\mathbb H$ attached to $\bar \infty$ ($v_{\bar \infty}(t)=-1$). Note that  $v_{\bar \infty}( f)\leq -1$ and, when $d_\infty>1,$   $\bar \infty$ can a priori belong  to the support of $V.$ We identify  the $\bar \infty$-adic  completion of $\mathbb H$ with  $ H((\frac{1}{t})).$ Therefore we deduce that:
$$f=\frac{\alpha(f)}{t^k}+\sum_{i\geq k+1} f_i \frac{1}{t^i}, k\leq -1$$
where $\alpha(f)\in H^\times,$  and $f_i\in H,$ for all $i\geq k+1.$

\medskip

Let $\exp_\phi$ be the unique element in $H\{\{\tau\}\}$ such that $\exp_\phi \equiv 1\pmod {\tau}$ and:
$$\forall a \in A, \quad \exp_\phi a=\phi_a\exp_\phi.$$
Write $\exp_\phi=\sum_{i\geq 0} e_i(\phi) \tau^i,$ then by \cite{GOS}, Corollary 7.4.9, we obtain:
$$H=K(e_i(\phi), i\geq 0).$$
Observe that $\exp_\phi$ induces an entire function on $\mathbb C_\infty,$ and there exists $\alpha\in \mathbb C_\infty^\times$ and $I \in \mathcal I(A)$ such that:
$$\forall z\in \mathbb C_\infty, \quad \exp_\phi(z)=\sum_{i\geq 0} e_i(\phi) z^{q^i}= z\prod_{a\in I\setminus\{0\}} (1- \frac{z}{\alpha a}).$$

Furthermore, we have (see for example \cite{THA2}, Proposition 0.3.6):
$$\forall i\geq 0, e_i(\phi)=\frac{1}{f\cdots f^{(i-1)}\mid_{\xi^{(i)}}}.$$
Thakur proved that if $e_n(\phi)=0,$ then $n\in \{ 2, \ldots, g-1\}$ (\cite{THA2}, proof of Theorem 3.2), and if $K$ has a place of degree one then $\forall n\geq 0, e_n(\phi)\not = 0.$ 

\medskip

Let $W(B)=\oplus_{i\geq 0} Bf\cdots f^{(i-1)}.$ Then $W(B)$ is a finitely generated $B\otimes_{\mathbb F_q}A=B[\rho(A)]$-module of rank one (see for example \cite{ANDTR1}, Lemma 4.4). Furthermore, 
$$\forall x\in W(B), \quad fx^{(1)}\in W(B).$$
Let $I \in \mathcal I(A).$ Let $\phi_I\in H\{\tau\}$ such that the coefficient of its term of highest degree in $\tau $ is one , and  such that:
$$\sum_{a\in I} H\{\tau\} \phi_a= H\{\tau \} \phi_I.$$
Then, we get:
\begin{align*}
\deg_\tau \phi_I &= \deg I,\\
{\rm Ker}\, \phi_I\mid_{\mathbb C_\infty} &= \cap_{a\in I} {\rm Ker}\, \phi_a\mid_{\mathbb C_\infty},\\
\phi_I &\in B\{\tau\}.
\end{align*}
We denote by $\psi_{\phi}(I) \in B\setminus \{0\}$ the constant term of $\phi_I.$ We set:
$$u_I= \sum_{j=0}^{\deg I} \phi_{I,j} f\cdots f^{(j-1)}\in W(B),$$
where $\phi_I=\sum_{j=0}^{\deg I} \phi_{I,j} \tau^j.$ 

\begin{lemma}\label{LemmaS2-1}
Let $I,J$ be two non-zero ideals of $A.$ We have:
\begin{align*}
u_I\mid_{\xi} &= \psi_{\phi}(I),\\
\sigma_I(f)u_I &= fu_I^{(1)},\\
u_{IJ} &= \sigma_I(u_J)u_I.
\end{align*}
\end{lemma}

\begin{proof} In \cite{ANDTR1}, Lemma 4.6,  we only gave a sketch of the proof of the above results. We give here a detailed proof for the convenience of the reader.\par
Observe that:
$$\forall i\geq 1, \quad (f\cdots f^{(i-1)})= V^{(i)}-V +\sum_{k=0}^{i-1} (\xi^{(k)})-\sum_{k=0}^{i-1}(\bar \infty^{(k)}).$$
Since $\xi$ does not belong to the support of $V,$ we deduce that:
$$u_I\mid_{\xi}=\psi_{\phi}(I).$$
Note that we have a natural isomorphism of $B$-modules:
 $$\gamma : \left\{\begin{array}{ccc}
 W(B)&\stackrel{\sim}{\longrightarrow}&  B\{\tau\}\\
 \forall i\geq 0,\ f\cdots f^{(i-1)}&\longmapsto& \tau^i.
 \end{array}
 \right.$$
For all $x\in W(B)$ and for all $a\in A$, we have:
\begin{align*}
\gamma(fx^{(1)}) &= \tau \gamma (x),\\
\gamma (\rho(a)x) &=\gamma(x)\phi_a.
\end{align*}

In particular $\gamma$ is an isomorphism of $B[\rho(A)]$-modules, and since $W(B)$ is a finitely generated $B[\rho(A)]$-module of rank one, this is also the case of $B\{\tau\}.$ Write $f=\frac{\sum_i \rho(a_i)b_i}{\sum_k \rho(c_k)d_k},$ for some $a_i, c_k\in A,$ $b_i, d_k\in B,$ we have the following equality in $B\{\tau\}:$
 $$\sum_ib_i\phi_{a_i}=\sum_k d_k \tau \phi_{c_k}.$$
 For $\sigma \in G, $   we set:
 $$W_\sigma (B)=\oplus_{i\geq 0} B\sigma(f)\cdots \sigma(f)^{(i-1)}.$$
 We have again an isomorphism of $B[\rho(A)]$-modules:
 $$\gamma_{\sigma} : W_\sigma (B)\simeq B\{\tau\}.$$
Again,
 $$\forall x\in W_{\sigma}(B), \forall a\in A, \quad \gamma_{\sigma} (\rho(a)x)=\gamma_{\sigma}(x)\phi^{\sigma}_a.$$
 
 
\medskip

Let $I$ be a non-zero ideal of $A,$ and let $\sigma=\sigma_I\in G.$  We start from the relation:
 $$\sum_ib_i^{\sigma}\phi^{\sigma}_{a_i}=\sum_k d_k^{\sigma} \tau \phi^{\sigma}_{c_k}.$$
 We multiply on the right by  $\phi_I, $ to obtain (see \cite{GOS}, Theorem 7.4.8):
$$\sum_ib_i^{\sigma}\phi_I\phi_{a_i}=\sum_k d_k^{\sigma} \tau\phi_I \phi_{c_k}.$$
Since $\gamma(fu_I^{(1)})= \tau \phi_I,$ we get:
 $$(\sum_i \rho(a_i) b_i^{\sigma}). \gamma(u_I)= (\sum_k d_k^{\sigma}\rho(c_k)). \gamma( fu_I^{(1)}).$$
In other words, we have proved:
$$\sigma(f)u_I= fu_I^{(1)}.$$

\medskip 

Now, let $J$ be a non-zero ideal of $A.$ We have:
$$\gamma(u_{IJ})= \phi_{IJ}= \phi^{\sigma}_J\phi_I.$$
Since $\forall i\geq 0, \sigma(f\cdots f^{(i-1)})u_I= f\cdots f^{(i-1)}u_I^{(i)},$ we get:
$$\gamma (u_J^{\sigma} u_I)=\phi^{\sigma}_J\phi_I.$$
It implies:
$$u_{IJ}= \sigma(u_J)u_I.$$

\end{proof}
\begin{corollary}\label{CorollaryS2-1}
We have:
$$\Sht=\{ \sigma(f), \sigma\in G\}.$$
Furthermore, for $\sigma \in G,$ $\phi^{\sigma}$ is the Drinfeld module associated to the shtuka function $\sigma(f).$
\end{corollary}
\begin{proof} Let $\sigma \in G$ and let $g\in \Sht$ be the shtuka function associated to $\phi^{\sigma}.$ By the proof of Lemma \ref{LemmaS2-1}, if   $a_i', c_k'\in A,$ $b_i', d_k'\in B$  are such that $\sum_ib_i'\phi^{\sigma}_{a_i'}=\sum_k d_k' \tau \phi^{\sigma}_{c_k'},$ then:
$$g=\frac{\sum_i \rho(a_i')b_i'}{\sum_k \rho(c_k')d_k'}.$$
Again, by the proof of Lemma \ref{LemmaS2-1}, we get:
$$g=\sigma(f).$$
\end{proof}

\begin{lemma}\label{LemmaS2-1.1}
Let $\iota_{\bar {\infty}}: \mathbb H\rightarrow H((\frac{1}{t}))$ be a  homomorphism of $\mathbb K$-algebras corresponding to $\bar{\infty}.$ Write $\iota_{\bar{\infty}}(f)=\frac{\alpha(f)}{t^k}+\sum_{i\geq k+1} f_i\frac{1}{t^i} \in H((\frac{1}{t})), \alpha(f)\in H^\times, f_i\in H, i\geq 0, k\leq -1.$ Then:
$$H=K(\mathbb F_\infty, \alpha(f), f_i, i\geq k+1).$$
Furthermore:
$$H_A=K(\mathbb F_\infty,  \frac{f_i}{\alpha(f)}, i\geq k+1).$$
In particular, there exists $u(f)\in B^\times$ such that:
\begin{itemize}
\item $H=H_A(u(f))$,
\item $\alpha(f)\equiv \iota_{\bar{\infty}}(u(f))\pmod{H_A^\times}$,
\item $\mathbb K(\frac{f}{u(f)})= {\rm Frac}(H_A\otimes _{\mathbb F_q}A)$.
\end{itemize}
\end{lemma}
\begin{proof} 
By Corollary \ref{CorollaryS2-1}, since $| G|= |\Sht|,$  we have:
$$\mathbb H=\mathbb K(f).$$
Recall that $H((\frac{1}{t}))$ is isomorphic to  the completion of $\mathbb H$ at $\bar \infty.$  Since $\infty$ splits totally in $K(\mathbb F_\infty)$ in $d_\infty$ places, we deduce that the natural map $\iota_{\bar{\infty}}: \mathbb H\hookrightarrow H((\frac{1}{t}))$ is ${\rm Gal}(H/ K(\mathbb F_\infty))$-equivariant. Thus:
$$H=K(\mathbb F_\infty, \alpha(f), f_i, i\geq k+1).$$
If $I=aA, a\in A\setminus\{0\},$ then $u_I = \rho(a)$, so that we have by Lemma \ref{LemmaS2-1} :
$$\sigma_I(f)=\sgn(a)^{q^{n(\phi)}-q^{n(\phi)+1}}f.$$
In particular:
$$\sgn_{\bar {\infty}^{(-1)}} (\iota_{\bar{\infty}^{(-1)}}(f))\not \in \mathbb F_\infty^\times.$$
We have $\alpha(f)^{\frac{q^{d_\infty}-1}{q-1}}\in H_A,$ and $\frac{f}{\alpha'(f)}\in {\rm Frac}(H_A\otimes _{\mathbb F_q}A),$ where $\alpha'(f)\in H^\times$ is such that $\iota_{\bar{\infty}}(\alpha'(f))=\alpha(f)$ (observe that $\iota_{\bar{\infty}} \mid_H \in G$). Since $\mathbb H=\mathbb K(f),$ we get the second  assertion.

\medskip

Since $H/H_A$ is totally ramified at each place of $H_A$ above $\infty,$ $\frac{B^\times}{(B)'^\times }$ is a finite abelian group, where we recall that $B'$ is the integral closure of $A$ in $H_A.$ Now recall that $H/H_A$ is a cyclic extension of degree $\frac{q^{d_\infty}-1}{q-1},$ and $\mathbb F_\infty \subset H_A.$ Let $\langle \sigma \rangle= {\rm Gal}(H_A((B)^\times)/H_A).$ Then we have an injective homomorphism:
$$\frac{B^\times}{(B')^\times }\hookrightarrow \mathbb F_\infty^\times, x\mapsto \frac{x}{\sigma(x)}.$$
The image of this homomorphism is a cyclic group of order dividing $\frac{q^{d_\infty}-1}{q-1}.$ By the proof \cite{GOS}, Theorem 7.6.4, there exists $\zeta\in \mathbb C_\infty^\times, \zeta^{q-1} \in H,$ such that:
$$\forall a\in A\setminus \{0\}, \zeta\phi_a\zeta^{-1}\in B'\{\tau\} \, {\rm and\,  its\,  highest \, coefficient\,  is\,  in\, } (B')^\times.$$
Thus $\zeta^{q-1}\in B^\times$ and $H=H_A(\zeta^{q-1}).$ In particular, there exists a group isomorphism:
$$ \frac{B^\times}{(B')^\times }\simeq \frac{\mathbb F_\infty^\times}{\mathbb F_q^\times}.$$
This implies by Kummer Theory that:
$$\alpha(f)\equiv u'(f) \pmod{H_A^\times},$$
for some $u'(f)\in  B^\times$ that generates the cyclic group $ \frac{B^\times}{(B')^\times }.$ Now define $u(f)$ to be the element in
$B^\times$ such that $\iota_{\bar{\infty}}(u(f))=u'(f).$
\end{proof}

\subsection{Special functions}${}$\par


We fix $\sqrt[q^{d_\infty}-1]{-\pi}\in \mathbb C_\infty$ a root of the polynomial $X^{q^{d_\infty}-1}+\pi=0.$ We consider the period lattice of $\phi$:
$$\Lambda (\phi)= \{ x\in \mathbb C_\infty,\exp_{\phi}(x)=0\}.$$
Then $\Lambda (\phi)$ is a finitely generated $A$-module of rank one and we have an exact sequence of $A$-modules induced by $\exp_{\phi}:$
$$0\rightarrow\Lambda (\phi) \rightarrow \mathbb C_\infty \rightarrow \phi(\mathbb C_\infty)\rightarrow 0,$$
where $\phi(\mathbb C_\infty)$ is the $\mathbb F_q$-vector space $\mathbb C_\infty$ viewed as an $A$-module via $\phi.$

\begin{lemma}\label{LemmaS2-2} 
We have:
$$\Lambda(\phi) \subset \sqrt[q^{d_\infty}-1]{-\pi}^{-q^{n(\phi)}}K_\infty,$$
and for all $I \in \mathcal I(A):$
$$\Lambda(\phi^{\sigma_I})= \psi_{\phi}(I)I^{-1} \Lambda(\phi).$$
\end{lemma}

\begin{proof} Observe that $\Lambda(\phi) K_\infty$ is a $K_\infty$-vector space of dimension one. Let $J$ be a non-zero ideal of $A,$ and let $\lambda_J\not =0$ be a generator of the $A$-module of $J$-torsion points of $\phi.$  By the proof of \cite{GOS}, Proposition 7.5.16, we have:
$$\lambda_J\in \Lambda(\phi)K_\infty.$$
By class field theory (see \cite{GOS}, section 7.5), we have:
$$E:=H(\lambda_J)\subset K_\infty(\sqrt[q^{d_\infty}-1] {-\pi}).$$
Furthermore, by \cite{GOS}, Remark 7.5.17, $$\lambda_J^{q^{d_\infty}-1}\in K_\infty^\times.$$
By local class field theory, for $x\in K_\infty^\times,$ we have:
$$(x,K_{\infty}(\sqrt[q^{d_\infty}-1] {-\pi})/K_\infty)(^{q^{dÐ\infty}-1}\sqrt {-\pi})= \frac{\sqrt[q^{d_\infty}-1] {-\pi}}{\sgn(x)}.$$
By \cite{GOS}, Corollary 7.5.7, for all $a\in K^\times, a\equiv 1\pmod{J},$ we get:
$$(aA,  E/K)(\lambda_J)= \sgn(a)^{-q^{n(\phi)}} \lambda_J.$$
Thus, for all $a\in K^\times, a\equiv 1\pmod{J}:$
$$ (a,K_{\infty}(\sqrt[q^{d_\infty}-1] {-\pi})/K_\infty)(\lambda_J)= \sgn(a)^{q^{n(\phi)}} \lambda_J.$$
Therefore, by the approximation Theorem, we get:
$$\forall x\in K_\infty^\times, \quad (x,K_{\infty}(\sqrt[q^{d_\infty}-1] {-\pi})/K_\infty)(\lambda_J)= \sgn(x)^{q^{n(\phi)}} \lambda_J.$$
It implies:
$$\lambda_J\in \sqrt[q^{d_\infty}-1]{-\pi}^{-q^{n(\phi)}}K_\infty.$$
Hence,
$$\Lambda(\phi)\subset \sqrt[q^{d_\infty}-1]{-\pi}^{-q^{n(\phi)}}K_\infty.$$

\medskip 

The second assertion comes from the fact that we have the following equality in $H\{\{\tau\}\}:$
$$\phi_I \exp_{\phi}= \exp_{\phi^{\sigma_I}} \psi_{\phi}(I).$$
\end{proof}

Set:
$$L= \rho(K)(\mathbb F_\infty)((\sqrt[q^{d_\infty}-1]{-\pi})).$$
Then, by the above Lemma,  $H\subset \mathbb F_\infty((\sqrt[q^{d_\infty}-1]{-\pi}))\subset L .$
Let $v_\infty:L\rightarrow \mathbb Q\cup\{+\infty\}$ be the valuation on $L$ which is trivial on $\rho(K)(\mathbb F_\infty)$ and such that $v_\infty(\sqrt[q^{d_\infty}-1]{-\pi})= \frac{1}{q^{d_\infty}-1}.$
Let $\tau: L\rightarrow L$ be the continuous homomorphism of $\rho(K)$-algebras such that:
$$\forall x\in \mathbb F_\infty((\sqrt[q^{d_\infty}-1]{-\pi})), \quad \tau (x)=x^q.$$
Observe that:
$$\forall x\in L, \quad v_\infty(\tau(x))=qv_\infty(x).$$

\begin{lemma}\label{LemmaS2-3}
We have:
$${\rm Ker} \exp_{\phi}\mid_L=  \Lambda(\phi)\rho(K),$$
where $\Lambda(\phi)\rho(K)$ is the $\rho(K)$-vector space generated by $\Lambda(\phi).$
\end{lemma}

\begin{proof} The proof is  standard in non-archimedean functional analysis, we give a sketch of the proof for the convenience of the reader. We have:
$$\Lambda(\phi)\rho(K)\subset {\rm Ker} \exp_{\phi}\mid_L.$$
Let:
$$\mathfrak M=\sqrt[q^{d_\infty}-1]{-\pi}\rho(K)(\mathbb F_\infty)[[\sqrt[q^{d_\infty}-1]{-\pi}]].$$
Let $\log_\phi \in H\{\{\tau\}\}$ such that $\log_\phi\exp_\phi=\exp_\phi\log_\phi=1.$ If we write: $\log_\phi=\sum_{i\geq 0} l_i(\phi)\tau^i,$ then there exists $C\in \mathbb R$ such that, for all $i\geq 0,$  $v_\infty(l_i(\phi))\geq C q^i.$ It implies that there exists an integer $N\geq 0$ such that $\exp_{\phi}$ is an isometry on $\mathfrak M^N.$  

\medskip 

Now, select $\theta\in A\setminus \mathbb F_q.$ Then:
$${\rm Ker}\exp_{\phi}\mid_{\mathbb F_\infty[\rho(\theta)]((\sqrt[q^{d_\infty}-1]{-\pi}))}=\Lambda(\phi)\mathbb F_q[\rho(\theta)].$$
Since $\rho(A)$ is finitely generated and free as an $\mathbb F_q[\rho(\theta)]$-module, it implies:
$${\rm Ker}\exp_\phi\mid_{\rho(A)[\mathbb F_\infty]((\sqrt[q^{d_\infty}-1]{-\pi}))}=\Lambda(\phi)\rho(A).$$
Let $V$ be the $\rho(K)$-vector space generated by $\rho(A)[\mathbb F_\infty]((\sqrt[q^{d_\infty}-1]{-\pi})).$ Then:
$${\rm Ker}\exp_{\phi}\mid_V=\Lambda(\phi)\rho(K).$$
Let $x\in {\rm Ker} \exp_{\phi}\mid_L,$ then there exists $y\in V$ such that:
$$x-y\in \mathfrak M^N.$$
Thus,
$$\exp_{\phi}(y-x)=\exp_{\phi}(y) \in \mathfrak M^N\cap V= \exp_{\phi}(\mathfrak M^N \cap V).$$
Therefore, $y=z+v,$ for some $z\in \mathfrak M^N \cap V,$ and some $v\in \Lambda(\phi)\rho(K).$ It implies that $x-v\in \mathfrak M^N,$ and hence:
$$x=v\in  \Lambda(\phi)\rho(K).$$
\end{proof}

\begin{lemma}\label{LemmaS2-4} We consider the following $\rho(K)$-vector space:
$$V=\bigcap_{a\in A\setminus \mathbb F_q}\exp_{\phi}(\frac{1}{a-\rho(a)}\Lambda(\phi)\rho(K))$$
Then, we have:
$${\rm dim}_{\rho(K)} V=1.$$
\end{lemma}

\begin{proof} For any $a \in A,$ we set:
$$V_a=\{ x\in L, \phi_a(x)= \rho(a)x \}.$$
Then, if $a\not \in \mathbb F_q,$ by Lemma \ref{LemmaS2-3}, we have:
$$V_a=\exp_{\phi}(\frac{1}{a-\rho(a)}\Lambda(\phi)\rho(K)),$$
and:
$${\rm dim}_{\rho(K)}V_a= \deg a= [K: \mathbb F_q(a)].$$

\medskip

Select $\theta\in A\setminus \mathbb F_q$ such that $K/\mathbb F_q(\theta)$ is a finite separable extension. Let $b\in A\setminus \mathbb F_q$ and let $P_b(X) \in \mathbb F_q[\theta][X]$ be the minimal polynomial of $b$ over $\mathbb F_q(\theta).$  Since $V_\theta$ is an $A$-module via $\phi$ and $\phi_b$ induces a $\rho(K)$-linear endomorphism of $V_\theta,$ it follows that:
$$\rho(P_b)(\phi_b)=0.$$
This implies that the minimal polynomial of $\phi_b$ viewed as an $\mathbb F_q(\rho(\theta))$-linear endomorphism of $V_\theta$ is $\rho(P_b(X)).$ Observe that $V_\theta$ is the $\rho(K)$-vector space generated by:
$$\exp_{\phi}(\frac{1}{\theta-\rho(\theta)}\Lambda(\phi)\mathbb F_q(\rho(\theta))),$$
and:
$${\rm dim}_{\mathbb F_q(\rho(\theta))}\exp_{\phi}(\frac{1}{\theta-\rho(\theta)}\Lambda(\phi)\mathbb F_q(\rho(\theta)))=\deg \theta.$$
Therefore, $\rho(P_b(X))$ is the minimal polynomial of $\phi_b$ viewed as a $\rho(K)$-linear endomorphism of $V_\theta.$ 

\medskip

Select $\theta'\in A\setminus \mathbb F_q$ such that $K=\mathbb F_q(\theta, \theta').$  Then the characteristic polynomial of $\phi_{\theta'}$ on the $\rho(K)$-vector space $V_\theta$ is $\rho(P_{\theta'}(X)).$ Since $P_{\theta'}(X)$ has simple roots, if $V'=V_{\theta}\cap V_{\theta'},$ we get:
$${\rm dim}_{\rho(K)}V'=1.$$
Now, let $b\in A,$ there exists $x, y\in A[\theta, \theta'],$ such that $b=\frac{x}{y}.$ Let $\lambda_b \in \rho(K)$ such $\phi_b \mid V'$ is the multiplication by $\lambda_b,$ then for any $v \in V' \setminus 0,$ we have:
$$\rho(y)\lambda_b v = \phi_{yb} v= \rho(x) v$$
It follows that:
$$\lambda_b=\rho(b).$$
\end{proof}

Let $\sgn:\rho(K)(\mathbb F_\infty)((\pi))^\times\rightarrow \rho(K)(\mathbb F_\infty)^\times$ be the group homomorphism such that ${\rm Ker} \sgn = \pi^{\mathbb Z}\times (1+\pi\rho(K)(\mathbb F_\infty)[[\pi]]),$ and $\sgn\mid_{\rho(K)(\mathbb F_\infty)^\times}={\rm Id}\mid_{\rho(K)(\mathbb F_\infty)^\times}.$ Let $\pi_* =(\sqrt[q^{d_\infty}-1]{-\pi})^{(q-1)q^{n(\phi)}}.$

\begin{lemma}\label{LemmaS2-5} We have:
\begin{align*}
f\pi_* &\in \rho(K)(\mathbb F_\infty)((\pi)),\\
v_\infty(f)&\equiv -\frac{(q-1)q^{n(\phi)}}{q^{d_\infty}-1}\pmod{(q-1)\mathbb Z},
\end{align*}
 and:
$$N_{\rho(K)(\mathbb F_\infty)/\rho(K)}(\sgn(f\pi_*))=1.$$
\end{lemma}

\begin{proof}${}$\par
\noindent 1) Recall that:
$$V=\bigcap_{a\in A\setminus \mathbb F_q}\exp_{\phi}(\frac{1}{a-\rho(a)}\Lambda(\phi)\rho(K)).$$
By Lemma \ref{LemmaS2-2}, we have:
$$V\subset (\sqrt[q^{d_\infty}-1]{-\pi})^{-q^{n(\phi)}} \rho(K)(\mathbb F_\infty)((\pi)).$$
Thus, by Lemma \ref{LemmaS2-4}, there exists  $U\in (\sqrt[q^{d_\infty}-1]{-\pi})^{-q^{n(\phi)}}\rho(K)(\mathbb F_\infty)((\pi))\setminus \{0\},$  such that:
$$\forall a\in A, \quad \phi_a(U)=\rho (a)U.$$
Write $f=\frac{\sum_i\rho(a_i)b_i}{\sum_k \rho(a_k')b_k'},$ $a_i, a'_k\in A,$ $b_i, b'_k\in B.$ Then, by the proof of Lemma \ref{LemmaS2-1}, we have:
$$\sum_ib_i\phi_{a_i}=\sum_k b'_k \tau \phi_{a'k}.$$
Thus,
$$(\sum_i\rho(a_i)b_i)U= (\sum_k \rho(a_k')b_k')\tau (U).$$
Therefore:
$$\tau(U)= fU.$$
In particular, 
$$\{ x\in L, \tau(x)=fx\}= \rho(K)U.$$
We also get:
$$f\in \pi_*^{-1} \rho(K)(\mathbb F_\infty)((\pi)).$$

\medskip 

\noindent 2)  Let $F=f\pi_*\in \rho(K)(\mathbb F_\infty)((\pi)).$ Set $R= U\,  (\sqrt[q^{d_\infty}-1]{-\pi})^{q^{n(\phi)}}\in \rho(K)(\mathbb F_\infty)((\pi)).$
We have:
$$\tau (R)=F R.$$
Let $i_0=v_\infty(F)\in \mathbb Z ,$ and write:
$$F=\sum_{i\geq i_0} F_i (-\pi) ^i, F_i \in \rho(K)(\mathbb F_\infty).$$
Let $\lambda= F_{i_0}.$ Set:
$$\alpha= \sqrt[q-1]{-\pi}^{i_0}(\prod_{i\geq 0} \frac{F^{(i)}}{\lambda^{(i)}(-\pi)^{q^i i_0}})^{-1}\in L^\times,$$
where $\sqrt[q-1]{-\pi}= (\sqrt[q^{d_\infty}-1]{-\pi})^{\frac{q^{d_\infty}-1}{q-1}}.$ Then clearly:
$$\tau(\alpha)=\frac{F}{\lambda}\alpha.$$
Thus:
$$\tau(\frac{R}{\alpha})= \lambda \frac{R}{\alpha}.$$
This implies:
$$R=\mu \alpha, \mu \in \rho(K)(\mathbb F_\infty)^\times.$$
In particular, $i_0\equiv 0\pmod{q-1},$ i.e. $v_\infty(f)\equiv - \frac{(q-1)q^{n(\phi)}}{q^{d_\infty}-1}\pmod{q-1}.$ Also:
$$\sgn(R)= \mu \sgn(\alpha).$$
Since $\sgn(\alpha) = (-1)^{\frac{i_0}{q-1}},$ we get:
$$\frac{\tau(\mu)}{\mu}= \lambda.$$
\end{proof}

We set: 
$$\mathbb T:=\rho(A)[\mathbb F_\infty]((\sqrt[q^{d_\infty}-1]{-\pi}))\subset L.$$ 
Then $\mathbb T$ is complete with respect to the valuation $v_\infty,$ and:
 $$\{x\in \mathbb T, \tau(x)=x\}=\rho(A).$$
Furthermore, we have (see the proof of Lemma \ref{LemmaS2-3}):
$${\rm Ker} \exp_\phi\mid_{\mathbb T}= \Lambda(\phi)\rho(A).$$
Let $\ev :\rho(A)[\mathbb F_\infty]\rightarrow \overline{\mathbb F}_q\subset \mathbb C_\infty$ be a homomorphism of $\mathbb  F_\infty$-algebras.  Such a homomorphism  induces a continuous homomorphism $\mathbb F_\infty((\sqrt[q^{d_\infty}-1]{-\pi}))$-algebras: $$\ev:  \mathbb  T \rightarrow \mathbb C_\infty.$$ We denote by $\mathcal E$ the set of such continuous homomorphisms from $\mathbb T$ to $\mathbb C_\infty.$
\begin{proposition}\label{PropositionS2-1} We have:
$$f\in \mathbb T^\times,$$
$$\sgn(f\pi_*)\in \rho(A)[\mathbb F_\infty]^\times.$$
Furthermore there exists $U\in \mathbb T\setminus \{0\}$ such that:
$$\{ x\in L, \tau (x)=fx\}= U\rho(K).$$
If $d_\infty=1,$ then $\sgn (f\pi_*)=1,$ and we can take:
$$U=\sqrt[q^{d_\infty}-1]{-\pi}^{-1}\, \sqrt[q-1]{-\pi}^{i_0}(\prod_{i\geq 0} \frac{(f\pi_*)^{(i)}}{(-\pi)^{q^i i_0}})^{-1} \in \mathbb T^\times,$$
where $i_0:=v_\infty(f\pi_*).$ 
\end{proposition}

\begin {proof} Recall that $f\in \mathbb H\subset L.$  Le $P$ be a point in $\bar X(\overline{\mathbb F}_q)$ above a maximal ideal of $\rho(A).$  Then $P$ is above a maximal ideal of $\rho(A)[\mathbb F_\infty]$ which can be viewed as the kernel of some homomorphism of $\mathbb F_\infty$-algebras $\ev:\rho(A)[\mathbb F_\infty]\rightarrow \overline{\mathbb F}_q.$  Since the field of constants of $H$ is $\mathbb F_\infty,$ we deduce that $\ev$ can be  uniquely extended to a homomorphism of $H$-algebras:
$$\ev: \rho(A)[H]\rightarrow \mathbb C_\infty.$$
Furthermore, the kernel of the above homomorphism corresponds to $P\cap \mathbb H$ (recall that $\mathbb H={\rm Frac}(\rho(A)[H])$).
Then $\ev$ extends to a continuous homomorphism of  $\mathbb F_\infty((\sqrt[q^{d_\infty}-1]{-\pi}))$-algebras: $$\ev:  \mathbb  T \rightarrow \mathbb C_\infty.$$
We deduce that, by \cite{THA2}, Lemma 1.1, for any $\ev\in \mathcal E,$ $\ev(f)$ is well-defined.
Thus  $f\in \mathbb T.$ Therefore, by Lemma \ref{LemmaS2-5}, we have:
$$f\in \pi_*^{\mathbb Z} \times (\sgn(f\pi_*)+\pi\rho(A)[\mathbb F_\infty][[\pi]]),$$
where $\sgn(f\pi_*)\in \rho(A)[\mathbb F_\infty]$ is such that:
$$N_{\rho(K)(\mathbb F_\infty)/\rho(K)}(\sgn(f\pi_*))=1.$$
Thus:
$$\sgn(f\pi_*)\in \rho(A)[\mathbb F_\infty]^\times,$$
and  there exists $\mu\in \rho(A)[\mathbb F_\infty]\setminus\{0\}$  such that:
$$\sgn(f\pi_*)=\frac{\tau(\mu)}{\mu}.$$
In particular, $f\in \mathbb T^\times.$ Furthermore, there exists a non-zero ideal $I$ of $A$ such that:
 $$\mu \rho(A)[\mathbb F_\infty]= \rho(I) \rho(A)[\mathbb F_\infty].$$
 Now, we use the proof of Lemma \ref{LemmaS2-5}. We put $i_0=v_\infty(f\pi_*)$ (observe that $i_0 \equiv 0\pmod{q-1}$) and set:
 $$U = \mu \alpha \, \sqrt[q^{d_\infty}-1]{-\pi}^{-q^{n(\phi)}},$$ where :
 $$\alpha =\sqrt[q-1]{-\pi}^{i_0}(\prod_{i\geq 0} \frac{(f\pi_*)^{(i)}}{\sgn(f\pi_*)^{(i)}(-\pi)^{q^i i_0}})^{-1}
 \in \mathbb T^\times.$$
 Then:
 $$\tau(U)=fU,$$
 $$U\in \mathbb  T.$$
 Note that $U$ is well-defined modulo $\rho(K)^\times$ and if $d_\infty=1,$ then $U\in \mathbb T^\times.$
\end{proof}

\begin{definition}
A non-zero element in $\{x\in L,\tau(x)=fx\}$ will be called a \emph{special function} attached to the shtuka function $f.$
\end{definition}

\begin{remark}\label{RemakS2-1} Let $M=\{x\in \mathbb T,\tau(x)=fx\}.$ Then, by the above Proposition,  there exists $U\in \mathbb T\setminus\{0\}$ such that:
$$U\rho(A)\subset M\subset U\rho(K).$$
Furthermore (see the proof of Lemma \ref{LemmaS2-5}):
$$M=\bigcap_{a\in A\setminus\mathbb F_q} \exp_\phi(\frac{1}{a-\rho(a)}\Lambda(\phi) \rho(A)).$$
Thus $M$ is a finitely generated $\rho(A)$-module of rank one. When $d_\infty=1,$ the above Proposition tells us that $M$ is a free $\rho(A)$-module. In general, we have:
$$M=U'\rho(\mathcal B),$$
where $\mathcal B\in \mathcal I(A), U'\in L^\times,$ and  $M=U''\rho(\mathcal B')$ if and only if $U'= xU''$ where $x\in \rho(K)^\times $ is such that $x\mathcal B=\mathcal B'.$\par
Let $I$ be a non-zero ideal of $A,$ and let   $\sigma=\sigma_I \in G.$ Recall that, by Lemma \ref{LemmaS2-1}, we have:
$$\sigma(f) =f\frac{\tau(u_I)}{u_I}.$$
Now observe that $u_I\in \mathbb T,$ $\frac{\tau(u_I)}{u_I}\in \mathbb T^\times,$ but in general we don't have $u_I\in \mathbb T^\times.$ By Lemma \ref{LemmaS2-1}, we have:
$$\frac{u_I} {\rho(x_I)}\in \mathbb T^\times,$$
where $I^n= x_IA,$ $n$ being the order of $I$ in ${\rm Pic}(A).$
Thus:
$$M_{\sigma}:=\{x\in \mathbb T,\tau(x)=\sigma(f)x\}=\frac{\rho(x_I)}{u_I}M.$$
We leave open the following question: is $M$ a free $\rho(A)$-module ? We will show in section \ref{sectiong} that the answer is positive if $g=0.$
\end{remark}
\subsection{The period $\tilde \pi$} ${}$\par
By Lemma \ref{LemmaS1-2}, and Lemma \ref{LemmaS2-2}, let $f$ be the  unique  shtuka function in $\Sht$ such that,  if $\phi$ is the Drinfeld module associated to $f,$ we have:
$${\rm Ker} \exp_\phi\mid_ L =\widetilde{\pi} A[\rho(A)],$$
where $\widetilde{\pi} \in\sqrt[q^{d_\infty}-1]{-\pi}^{-q^{n(\phi)}}K_\infty,$  $\sgn(\widetilde{\pi}\,(  \sqrt[q^{d_\infty}-1]{-\pi})^{q^{n(\phi)}})=1.$
\begin{proposition}\label{PropositionS2-2} There exists $\theta\in A\setminus\mathbb F_q,$ $a\in A[\rho(A)] ,$ and a special function $U\in \mathbb T,$ such that for all $i\geq 0 :$
$$\frac{\rho(\theta)-\theta^{q^i}}{a^{(i)}}U\mid_{\xi^{(i)}}= e_i(\phi)\widetilde{\pi}^{q^i}.$$
In particular, for any special function $U'$ associated to $f,$ we have :
$$\forall i\geq 0, \quad f^{(i)}U'\mid_{\xi^{(i)}}\in \widetilde{\pi}^{q^i} H.$$
\end{proposition}
\begin{proof}

Let $ \mathbb A= A[\rho(K)].$  We  still denote by $\rho$ the obvious $\rho(K)$-linear map $\mathbb A\rightarrow \rho(K).$ We observe that:
$${\rm Ker}\rho =\sum_{a\in A} (a-\rho(a))\mathbb A.$$
We also observe that there exists $\theta\in A\setminus\mathbb F_q$ such that $\rho(\theta)-\theta\in {\rm Ker}\rho \setminus ({\rm Ker}\rho)^2.$ Set $z= \rho(\theta).$  Then   $z-\theta$ has a zero of order  one at $\xi$ (observe that $z-\theta^{q^i}$ has a zero of order one  at $\xi^{(i)}$). Note that $K/\mathbb F_q(\theta)$ is a finite separable extension, therefore there exists $y\in A$ such that $K=\mathbb F_q(\theta, y).$ Let $P(X)\in \mathbb F_q[\theta][X]$ be the minimal polynomial of $y$ over $\mathbb F_q(\theta)$ and set:
$$a=\frac{P(X)}{X-y}\mid_{X=\rho(y)}\in A[\rho(A)]\subset \mathbb A.$$
Since $P(X)$ has a zero of order one at $y,$ we have:
$$a\not \in {\rm Ker}\rho.$$
Let's set:
$$U=\exp_\phi(\frac{a}{z-\theta}\widetilde{\pi})\in \mathbb T.$$
Since $\frac{a}{z-\theta}\not \in \mathbb A,$ we have:
$$U\not =0.$$
Furthermore, observe that $\mathbb F_q[\theta, y]\subset A\subset {\rm Frac}(\mathbb F_q[\theta,y]).$ Thus:
$$\forall b\in A, \quad \phi_b(U)=\rho(b)U.$$
We conclude that:
$$U\in (\{ x\in L, \tau(x)=fx\}\setminus\{0\})\cap \mathbb T.$$
Let's set:
$$\delta=\frac{a}{z-\theta}.$$
We have:
$$U=\sum_{i\geq 0} \delta^{(i)} e_i(\phi) \widetilde{\pi}^{q^i}.$$
We therefore get:
$$\forall i\geq 0, \quad (\delta^{-1})^{(i)}U\mid_{\xi^{(i)}}= e_i(\phi)\widetilde{\pi}^{q^i}.$$

 The last assertion comes from the fact that $f^{(i)}$ has a zero of order at least  one at $\xi^{(i)}.$

 \end{proof}
 We refer the reader to \cite{AND} for the explicit construction of $f$ in the case $d_\infty=1,$ and to \cite{GRE&PAP} for the explicit construction of the special functions attached to $f$ in the case $g=1$ and $d_\infty=1.$
 \section{A basic example: the case $g=0$}\label{sectiong}
 
 \medskip 
 
In this section, we assume that the genus of $K$ is zero. Let's select $x\in K$ such that $K=\mathbb  F_q(x)$ and $v_\infty(x)=0.$ Let $P_\infty(x)\in \mathbb  F_q[x]$ be the monic irreducible polynomial corresponding to $\infty,$ then $\deg_xP_\infty(x)=d_\infty.$ Let $\sgn:K_\infty^\times \rightarrow \mathbb F_\infty^\times$ be the sign function such that $\sgn(P_{\infty}(x))=1.$ Then:
 $$A=\{ \frac{f(x)}{P_\infty(x)^k}, k\in \mathbb N, f(x)\in \mathbb F_q[x], f(x)\not \equiv 0\pmod{P_\infty(x)}, \deg_x(f(x))\leq kd_\infty\}.$$
 Observe that:
 $${\rm Pic}(A)\simeq \frac{\mathbb Z}{d_\infty\mathbb Z}.$$
 Let $P$ be the maximal ideal of $A$ which corresponds to the pole of $x,$ i.e. $P=\{\frac{f(x)}{P_\infty(x)^k}, k\in \mathbb N, f(x)\in \mathbb F_q[x], f(x)\not \equiv 0\pmod{P_\infty(x)}, \deg_x(f(x))< kd_\infty\},$ the order of $P$ in ${\rm Pic}(A)$ is exactly $d_\infty,$ and $P^{d_\infty}=\frac{1}{P_\infty(x)} A.$ We also observe that the Hilbert class field of $A$ is $K(\mathbb F_\infty).$  Let $\zeta= \sgn(x) \in \mathbb F_\infty^\times.$ Then $P_{\infty}(\zeta)=0.$  Note that:
 $$v_\infty(x-\zeta)=1,$$
 $$\sgn(x-\zeta)= P_{\infty}'(\zeta)^{-1}.$$
 The integral closure of $A$ in $K(\mathbb F_\infty)$ is $A[\mathbb F_\infty].$ The abelian group $A[\mathbb F_\infty]^\times $ is equal to:
 $$\mathbb F_\infty^\times \prod_{k=1}^{d_\infty-1}(\frac{x-\zeta}{x-\zeta^{q^k}})^{\mathbb Z}.$$
 We know that $A[\mathbb F_\infty]$ is a principal ideal domain and we have:
 $$PA[\mathbb F_\infty]= \frac{1}{x-\zeta}A[\mathbb F_\infty].$$
 Furthermore $B=A[\mathbb F_\infty][u],$ where $u\in B^\times$ is such that:
 $$u^{\frac{q^{d_\infty}-1}{q-1}}=\prod_{k=0}^{d_\infty-1} \frac{\zeta -x^{q^k}}{\zeta^{q^k}-x^{q^k}}.$$
 Indeed, using Thakur Gauss sums (\cite{THA3}), there exists $g\in \overline{K}$ such that $K(\mathbb F_\infty, g)/K$ is a finite abelian extension and:
$$g^{q^{d_\infty}-1}= \prod_{k=0}^{d_\infty-1} (\zeta -x^{q^k}).$$
 Furthermore $K(\mathbb F_\infty,g)/K$ is unramified outside $\infty$  and the pole of $x,$ and $P_\infty(x)$ is a local norm for every place of $K(\mathbb F_\infty,g)$ above $\infty.$\par
 Let $z=\rho(x)\in \rho(K)^\times.$ Then:
 $$\mathbb H=H(z).$$
 Let $Q\in \bar X(\mathbb F_q)$ be the unique point which is a pole of $z,$ then:
 $$(z-x)=(\xi)-(Q).$$
 We choose $\bar {\infty}$ to be the point of $\bar X(\mathbb F_\infty)$ which is the zero of $z-\zeta.$ Then:
 $$(\frac{z-x}{z-\zeta})=(\xi)-(\bar \infty).$$
 We easily deduce that if $f$ is a shtuka function relative to $\bar {\infty}$ (note that $f$ is well-defined modulo $\{ x\in \mathbb F_\infty^\times, x^{\frac{q^{d_\infty}-1}{q-1}}=1\}$), then $f$ is of the form:
 $$\frac{z-x}{z-\zeta}v, v\in H^\times.$$
Let $\theta= \frac{1}{P_{\infty}(x)}\in A.$ Then:
$$\sgn(\theta)=1,$$
$$\deg \theta= d_\infty.$$
 Let $\phi$ be the Drinfeld module attached to $f,$ then:
 $$\phi_\theta= \theta +\cdots +  \tau^{d_\infty}.$$
 We have:
 $$f\cdots f^{(d_\infty-1)}= \frac{\prod_{k=0}^{d_\infty-1} (z-x^{q^k})}{P_\infty(z)} v^{\frac{q^{d_\infty}-1}{q-1}}.$$
 We get:
 $$1=\prod_{k=0}^{d_\infty-1} (\zeta-x^{q^k})v^{\frac{q^{d_\infty}-1}{q-1}}.$$
 Thus:
 $$(vg^{q-1})^{\frac{q^{d_\infty}-1}{q-1}}=1,$$
 So that,
 $$f=\frac{z-x}{z-\zeta}g^{1-q} \zeta',$$
 where $\zeta' \in \mathbb F_\infty^\times$ is such that:
 $$(\zeta')^{\frac{q^{d_\infty}-1}{q-1}}=1.$$
 Furthermore, if we write $\exp_\phi=\sum_{i\geq 0} e_i(\phi) \tau^i, e_i(\phi)\in H,$ then:
 $$e_i(\phi)= g^{q^i-1} (\zeta')^{-\frac{q^i-1}{q-1}}\prod_{k=0}^{i-1}\frac{x^{q^i}-\zeta^{q^k}}{x^{q^i}-x^{q^k}}.$$
 We also deduce that:
 $$\forall a\in A, \phi_a=a+\cdots + \sgn(a) \tau^{\deg a}.$$

 Recall that $H\subset \mathbb C_\infty,$ and $v_\infty(x-\zeta)=1.$  We now work in $$L=\mathbb F_\infty(z)((\sqrt[q^{d_\infty}-1]{-P_\infty(x)})).$$ Recall that $g$ is the Thakur-Gauss sum associated to $\sgn,$ i.e. let $C:\mathbb F_q[x]\rightarrow \mathbb F_q[x]\{\tau\}$ be the homomorphism of $\mathbb F_q$-algebras such that $C_x=x+\tau,$ we have chosen $\lambda\in H\setminus\{0\}$ such that $C_{P_\infty(x)}(\lambda)=0,$ and:
 $$g= -\sum_{y\in \mathbb F_q[x]\setminus\{0\}, \deg_xy<d_\infty} \sgn (y)^{-1}C_y(\lambda).$$
 Furthermore, $\lambda$ is chosen is such a way that:
 $$\lambda \in \sqrt[q^{d_\infty}-1]{-P_\infty(x)}K_\infty,$$
 $$\sgn(\frac{\lambda}{\sqrt[q^{d_\infty}-1]{-P_\infty(x)}})=1.$$
 Thus:
 $$\sgn(\frac{g}{\sqrt[q^{d_\infty}-1]{-P_\infty(x)}})=1.$$
 Recall also that:
 $$\mathbb T= \rho(A)[\mathbb F_\infty]((\sqrt[q^{d_\infty}-1]{-P_\infty(x)})). $$
 We can choose $f$ such that $\zeta'=1,$ i.e. $f=\frac{z-x}{z-\zeta} g^{1-q}.$ Now, recall  that:
 $$f, \frac{z-x}{z-\zeta}\in \mathbb T^\times.$$
Set:
$$U=\prod_{i\geq 0}(1+\frac{(\zeta-x)^{q^i}}{z-\zeta^{q^i}})^{-1}\in L^\times.$$
Then:
$$U \in \mathbb T^{\times}.$$
Furthermore:
$$\tau (U)= \frac{z-x}{z-\zeta}U.$$
Let's set:
$$\omega= g^{-1}U,$$
Then:
\begin{align*}
\tau (\omega)&=f \omega,\\
\sgn(\omega\, \sqrt[q^{d_\infty}-1]{-P_\infty(x)})&=1,\\
\omega &\in \mathbb T^\times,\\
\{ x\in \mathbb T, \tau (x)=fx\}&= \omega \rho(A).
\end{align*}
Finally observe that:
$$(z-x)\omega\mid_{\xi}= g^{-1}(x-\zeta)\prod_{i\geq 1}(1+\frac{(\zeta-x)^{q^i}}{x-\zeta^{q^i}})^{-1}.$$
Thus, there exists $b\in K^\times, $ $\sgn(b)=1,$ $\zeta'$ a root of $P_\infty(x),$ such that:
$$\widetilde{\pi}= bg'^{-1}(x-\zeta')\prod_{i\geq 1}(1+\frac{(\zeta'-x)^{q^i}}{x-(\zeta')^{q^i}})^{-1},$$
for some well-chosen Thakur Gauss sum $g'$ relative to a twist of $\sgn.$\par
Let's treat the elementary (and well-known, see \cite{AND&THA}, and especially the proof of Lemma 2.5.4) case $d_\infty=1.$Then $A=\mathbb F_q[\theta]$ for some $\theta\in K,$ $\sgn (\theta)=1.$ Let's take $x=\frac{\theta+1}{\theta}.$ Then $P_\infty(x)= x-1,$ and $\zeta=1.$ In that case:
$$g=\sqrt[q-1]{-P_\infty(x)} =\sqrt[q-1]{-\frac{1}{\theta}}.$$
We get:
$$f=\frac{z-x}{z-1}g^{1-q}=t-\theta ,$$
where $t=\rho(\theta).$ We have:
$$\phi_{\frac{1}{P_\infty(x)}}= \phi_\theta= \theta+\tau.$$
We get:
$$\omega= \sqrt[q-1]{-\theta}\prod_{i\geq 0} (1- \frac{t}{\theta^{q^i}})^{-1}\in \mathbb T= \mathbb F_q[t]((\sqrt[q-1]{\frac{-1}{\theta}})).$$
In this case $\phi$ is standard, thus we have:
$${\rm Ker}\exp_\phi=\widetilde{\pi} A,$$
for $\widetilde{\pi}\in \sqrt[q-1]{-\theta} K_\infty, \sgn(\widetilde{\pi} \, \sqrt[q-1]{\frac{-1}{\theta}})=1.$ Let's set:
$$\omega'=\exp_\phi(\frac{\widetilde{\pi}}{f})\in \mathbb T\setminus\{0\}.$$
Then, one has:
$$\phi_\theta(\omega')= \exp_\phi (\theta \frac{\widetilde{\pi}}{t-\theta})= t\omega'.$$
Thus:
$$\forall a\in A, \phi_a(\omega')= \rho(a)\omega'.$$
Therefore there exists $a\in A\setminus\{0\}$ such that:
$$\omega' =\omega \rho(a).$$
But, since $\forall i\geq 0,$ $v_\infty (e_i(\phi))=i q^i,$ by examining the Newton polygon of $\sum_{i\geq 0} e_i(\phi) \tau^i,$ we get:
$$v_\infty(\widetilde{\pi})= \frac{-q}{q-1}.$$
This implies:
$$v_\infty (\omega'- \frac{\widetilde{\pi}}{f})\geq q-\frac{q}{q-1} .$$
Therefore:
$$\sgn(\omega'\, \sqrt[q-1]{\frac{-1}{\theta}})=\sgn (\frac{\widetilde{\pi}}{f}\, \sqrt[q-1]{\frac{-1}{\theta}})=-1.$$
Thus:
$$\omega'=-\omega.$$
We get:
$$\frac{- \widetilde{\pi}}{\theta^2}= (z-x)\omega'\mid_{\xi}= -(z-x)\omega\mid_{\xi}.$$
Thus:
$$(z-x)\omega\mid_{\xi}= \frac{\widetilde{\pi}}{\theta^2},$$
and therefore:
$$ \widetilde{\pi}= \theta^2(z-x)\omega\mid_{\xi}= \sqrt[q-1]{-\theta} \theta \prod_{i\geq 1} (1-\theta^{1-q^i})^{-1}.$$
\section{A rationality result for twisted $L$-series}${}$\par
Let $s$ be an integer, $ s\geq 1.$ We introduce:
$$\mathcal A_s=A\otimes_{\mathbb F_q}\cdots \otimes_{\mathbb F_q} A = A^{\otimes s},$$
and set:
$$k_s={\rm Frac}(\mathcal A_s).$$
For $i=1, \ldots, s,$ let $\rho_i: K\rightarrow k_s$ be the homomorphism of $\mathbb F_q$-algebras such that $\forall a\in A, \rho_i(a)=1\otimes\cdots 1\otimes a\otimes 1\cdots \otimes 1,$ where $a$ appears at the $i$th position. We set:
$$\mathbb A_s =A\otimes_{\mathbb F_q}k_s,$$
$$\mathbb K_s={\rm Frac} (\mathbb A_s),$$
$$\mathbb H_s= {\rm Frac}(B\otimes_{\mathbb F_q} k_s).$$
We identify $H$ with its image $H\otimes 1$ in $\mathbb H_s,$ and $k_s$ with its image $1\otimes k_s.$ Thus:
$$\mathbb A_s=A[k_s].$$
We also identify $G$ with the Galois group of $\mathbb H_s/\mathbb K_s.$  For $i=1, \ldots, s,$ $\rho_i$ induces a homomorphism of $H$-algebras:
$$\rho_i: \mathbb H\rightarrow \mathbb H_s.$$
Let $\mathbb K_{s,\infty}$ be the $\infty$-adic completion of $\mathbb K_s,$ i.e.:
$$\mathbb K_{s, \infty}= k_s[\mathbb F_\infty]((\pi)).$$
We set:
$$\mathbb H_{s, \infty}= \mathbb  H_s\otimes_{\mathbb K_s} \mathbb K_{s, \infty}.$$
Then we have an isomorphism of $\mathbb K_{s, \infty}$-algebras:
$$\kappa: \mathbb H_{s, \infty}\simeq k_s[\mathbb F_\infty]((\pi_*))^{\mid{\rm Pic}(A)\mid},$$
where we set   $\pi_*:=\sqrt[\frac{q^{d_\infty}-1}{q-1}]{-\pi}.$\par
Let $V$ be a finite dimensional $\mathbb K_{s, \infty}$-vector space. An $\mathbb A_s$-module $M,$ $M\subset V,$ will be called an $\mathbb A_s$-lattice in $V,$ if $M$ is a finitely generated $\mathbb A_s$-module which is discrete in $V$ and such that $M$ contains a $\mathbb K_{s, \infty}$-basis of $V.$ For example, $\mathbb B_s:=B[k_s]$ is an $\mathbb A_s$-lattice in $\mathbb H_{s, \infty}.$

\medskip

Let $\phi\in \Drin$ and let $f$ be its associated shtuka function. For $i=1, \ldots, s$ we set:
$$f_i=\rho_i(f).$$
Let $\tau : \mathbb H_{s, \infty}\rightarrow \mathbb H_{s, \infty}$ be the continuous homomorphism of $k_s$-algebras such that:
$$\forall x\in H\otimes_KK_\infty, \quad \tau(x)=x^q.$$
Let $\varphi_s: \mathbb A_s\rightarrow \mathbb H_s\{\tau\}$ be the homomorphism of $k_s$-algebras such that:
$$\forall a\in A, \quad \varphi_{s,a}=\sum_{k=0}^{\deg a} \phi_{a,k} \prod_{i=1}^s\prod_{j=0}^{k-1} f_i^{(j)} \, \tau ^k.$$
We consider:
$$\exp_{\varphi_s} = \sum_{k\geq 0} e_k(\phi)\prod_{i=1}^s\prod_{j=0}^{k-1} f_i^{(j)} \, \tau ^k \in \mathbb H_s\{\{\tau\}\}.$$
Then:
$$\forall a\in \mathbb A_s, \quad \exp_{\varphi_s} a= \varphi_{s,a}\exp_{\varphi_s}.$$
Furthermore $\exp_{\varphi_s}$ converges on $\mathbb H_{s, \infty}.$

\begin{proposition}\label {PropositionS3-1} 
Assume that $s\equiv 1\pmod{q-1}.$ The $\mathbb  A_s$-module ${\rm Ker}(\exp_{\varphi_s}: \mathbb H_{s, \infty} \to \mathbb H_{s, \infty})$ is a finitely generated $\mathbb A_s$-module, discrete in $\mathbb H_{s,\infty}$ and of rank $\mid{\rm Pic}(A)\mid$. In particular, ${\rm Ker}\exp_{\varphi_s}$ is an $\mathbb A_s$-lattice in $\{x\in \mathbb H_{s, \infty}, \forall a\in A\setminus \{0\}, \sigma_{aA}(x)= \sgn (a)^{q^{n(\phi)}(s-1)}x\}.$
Furthermore, if $s\not \equiv 1\pmod{q-1},$ then:
$${\rm Ker}\exp_{\varphi_s}=\{0\}.$$
\end{proposition}

\begin{proof} One can show that, for any $s,$ ${\rm Ker}\exp_{\varphi_s}$ is a finitely generated $\mathbb A_s$-module and is discrete in $\mathbb H_{s, \infty}.$\par
We view $\mathbb H_s$ as a subfield of $k_s[\mathbb F_\infty]((\pi_*)).$ There exists $\mathcal G\subset G$ a system of representatives of $\frac{G}{{\rm Gal}(H/H_A)},$ such that:
$$\forall x\in \mathbb H_s, \quad \kappa(x)= (\sigma(x))_{\sigma\in \mathcal G}.$$
By Proposition \ref{PropositionS2-1}, for $i=1, \ldots, s,$ $\sigma \in \mathcal G,$ we can  select a non-zero element $U_{i, \sigma}\in L_s= k_s[\mathbb F_\infty]((\sqrt[q^{d_\infty}-1]{-\pi}))$ such that:
$$\tau(U_{i,\sigma})=\sigma(f_i)U_{i, \sigma}.$$
Thus, by similar arguments to those of the proof of Lemma \ref{LemmaS2-3}, we get:
$${\rm Ker}\exp_{\sigma(\varphi_s)}\mid_{L_s}= \frac{\Lambda(\phi^{\sigma})k_s}{\prod_{i=1}^sU_{i,\sigma}}.$$
Recall that (see Proposition \ref{PropositionS2-1}):
$$U_{i,\sigma}\in \Lambda(\phi^{\sigma})k_s\subset (\sqrt[q^{d_\infty}-1]{-\pi})^{-q^{n(\phi)}} \mathbb K_{s, \infty},$$
and (Lemma \ref{LemmaS2-2}):
$$\Lambda(\phi^{\sigma})k_s\subset (\sqrt[q^{d_\infty}-1]{-\pi})^{-q^{n(\phi)}} \mathbb K_{s, \infty} .$$
Thus:
$${\rm Ker}\exp_{\sigma(\varphi_s)}\mid_{L_s}\subset(\sqrt[q^{d_\infty}-1]{-\pi})^{q^{n(\phi)}(s-1)} \mathbb K_{s, \infty} .$$
Thus, if $s\equiv 1\pmod{q-1},$ we get:
$${\rm Ker}\exp_{\sigma(\varphi_s)}\mid_{k_s[\mathbb F_\infty]((\pi_*))}=\frac{\Lambda(\phi^{\sigma})k_s}{\prod_{i=1}^sU_{i,\sigma}},$$
and if $s\not \equiv 1\pmod{q-1}:$
$${\rm Ker}\exp_{\sigma(\varphi_s)}\mid_{k_s[\mathbb F_\infty]((\pi_*))}=\{0\}.$$

\end{proof}
\begin{remark} Let $\mathbb H'_s= {\rm Frac}(H_A\otimes_{\mathbb F_q} k_s).$ Let $I=aA, a\in A \setminus\{0\},$ and $\sigma=\sigma_I\in {\rm Gal}(H/H_A).$  We have already noticed  that:
$$\sigma(f) = \sgn(a)^{q^{n(\phi)}-q^{n(\phi)+1}}f.$$
We verify that:
$$\forall \sigma \in {\rm Gal}(H/H_A), \quad \varphi_s^{\sigma}= \varphi_s \Leftrightarrow s\equiv 1\pmod{\frac{q^{d_\infty}-1}{q-1}}.$$
In particular, when $s\equiv 1\pmod{q^{d_\infty}-1},$ $\varphi_s$ is defined over $\mathbb H'_s,$ $\exp_{\varphi_s}: \mathbb H_s\rightarrow \mathbb H_s$ is ${\rm Gal}(H/H_A)$-equivariant, and  ${\rm Ker}\exp_{\varphi_s}$ is an $\mathbb A_s$-lattice in $\mathbb H'_{s, \infty}:= \mathbb H'_s\otimes_{\mathbb K_s}\mathbb K_{s, \infty}.$
\end{remark}
We introduce (see \cite{ANDTR1}):
 $$\mathcal L_s=\sum_{I\in \mathcal I(A), I\subset A} \frac{\prod_{k=1}^s \rho_k(u_I)}{\psi_\phi(I)} \sigma_I\in \mathbb H_{s, \infty}[G]^\times.$$
\begin{theorem}\label{TheoremS3-1}
Let $s\equiv 1\pmod{\frac{q^{d_\infty}-1}{q-1}}.$  Set:
$$W'_s= (\oplus _{i_1, \ldots, i_s\geq 0} B\prod_{k=1}^s f_k\cdots f_k^{(i_k-1)})^{{\rm Gal}(H/H_A)}.$$
Then:
$$\exp_{\varphi_s}(\mathcal L_sW'_s)\subset W'_s.$$
\end{theorem}
\begin{proof} By our assumption on $s,$ and by Lemma \ref{LemmaS2-1}, we get:
$$\mathcal L_s\in \mathbb H'_{s, \infty}[G]^\times.$$
The result  is then a consequence of  the above remark and  \cite{ANDTR1}, Corollary 4.10.
\end{proof}
\begin{remark} Let $W_s'=(\oplus _{i_1, \ldots, i_s\geq 0} B\prod_{k=1}^s f_k\cdots f_k^{(i_k-1)})^{{\rm Gal}(H/H_A)}.$ By Lemma \ref{LemmaS2-1.1}, there exists $u\in B^\times$ such that:
$$\frac{f}{u}\in {\rm Frac}(H_A\otimes_{\mathbb F_q}A).$$
In particular:
$$B=B'[u],$$
where we recall that $B'$ is the integral closure of $A$ in $H_A.$ Thus:
$$W'_s=\oplus _{i_1, \ldots, i_s\geq 0}B' u^{-\sum_{k=1}^s \frac{q^{i_k}-1}{q-1}}\prod_{k=1}^s f_k\cdots f_k^{(i_k-1)}.$$
Let $\mathbb W'_s$ be the $k_s$-vector space generated by $W'_s.$ Then, by the proof of \cite{ANDTR1}, Lemma 4.4, $\mathbb W'_s$ is a fractional ideal of $\mathbb  B'_s:= B'[k_s],$ and therefore $\mathbb W'_s$ is an $\mathbb A_s$-lattice in $\mathbb H'_{s, \infty}.$
\end{remark}
\begin{proposition}\label{PropositionS3-2}
Let $s\equiv 1\pmod{\frac{q^{d_\infty}-1}{q-1}}.$ We set:
$$\mathbb U_s= \{ x\in \mathbb H'_{s, \infty}, \exp_{\varphi_s}(x)\in \mathbb W'_s\}.$$
Then $\mathbb  U_s$ is an $\mathbb A_s$-lattice in $\mathbb H'_{s, \infty}$ and:
$$\mathcal L_s \mathbb  W'_s \subset \mathbb U_s.$$
If furthermore $s\equiv 1\pmod{q^{d_\infty}-1},$ then $\frac{\mathbb U_s}{{\rm Ker} \exp_{\varphi_s}}$ is a finite dimensional $k_s$-vector space. In particular, there exists $a\in \mathbb A_s\setminus\{0\}$ such that:
$$a\mathcal L_s \mathbb  W'_s\subset {\rm Ker} \exp_{\varphi_s}.$$
\end{proposition}
\begin{proof} Since $\mathbb W'_s$ is  an $\mathbb A_s$-lattice in $\mathbb H'_{s,\infty},$ we deduce that $\mathbb U_s$ is discrete in $\mathbb H'_{s, \infty}$ and is a finitely generated $\mathbb A_s$-module. By Theorem \ref{TheoremS3-1}, we have:
$$\mathcal L_s \mathbb  W'_s \subset \mathbb U_s.$$
Let $G'={\rm Gal}(H_A/K),$ and let ${\rm res}: \mathbb H'_{s, \infty}[G]\rightarrow \mathbb H'_{s, \infty}[G']$ be the usual restriction map, then:
$${\rm res}(\mathcal L_s)\in \mathbb H'_{s, \infty}[G']^\times.$$
Therefore $\mathcal L_s \mathbb  W'_s $ is  an $\mathbb A_s$-lattice in $\mathbb H'_{s, \infty}.$ We conclude that $\mathbb U_s$ is an $\mathbb A_s$-lattice in $\mathbb H'_{s,\infty}.$\par
If $s\equiv 1\pmod{q^{d_\infty}-1},$ then ${\rm Ker} \exp_{\varphi_s}$ is a $\mathbb  A_s$-lattice in $\mathbb H'_{s, \infty}$ by Proposition \ref{PropositionS3-1}. The proposition follows.
\end{proof}
\begin{theorem}\label{TheoremS3-1}
Let $s\equiv 1\pmod{q^{d_\infty}-1}.$  We work  in $L_s:=k_s[\mathbb F_\infty]((\sqrt[q^{d_\infty}-1]{-\pi})).$ There exist non-zero elements $\omega_1, \ldots, \omega_s \in \mathbb T_s:=
\mathcal A_s[\mathbb F_\infty]((\sqrt[q^{d_\infty}-1]{-\pi}))$ such that:
$$\tau (\omega_i)=f_i\omega_i.$$
There also exists $h\in B\setminus\{0\}$ such that:
$$\forall x\in \mathbb W'_s, \quad
\frac{\mathcal L_s(x) \prod_{k=1}^s \omega_i}{\widetilde{\pi}}\in h \mathbb K_s.$$
Furthermore, if $\phi $ is standard, then $h\in \mathbb F_\infty^\times.$
\end{theorem}
\begin{proof} By Proposition \ref{PropositionS2-1}, we have:
$$f_1, \ldots, f_s\in \mathbb T_s^\times.$$
By the same proposition, there exist $\omega_1, \ldots, \omega_s \in \mathbb T_s\setminus\{0\}$ such that:
$$\tau (\omega_i)=f_i\omega_i.$$
We deduce,  by Lemma \ref{LemmaS2-2} and Lemma \ref{LemmaS2-3}, that:
$${\rm Ker} \exp_{\varphi_s}\mid_L= \frac{h\widetilde{\pi} I\mathbb A_s}{\prod_{k=1}^s \omega_i},$$
where $I$ is some factional ideal of $A,$ $h\in H^\times.$ Let $x\in \mathbb W'_s,$ by Proposition \ref{PropositionS3-2}, we get:
$$\frac{\mathcal L_s(x) \prod_{k=1}^s \omega_i}{\widetilde{\pi}}\in h \mathbb K_s.$$
\end{proof}

We end this section with an application of the above Theorem. Let $\phi\in \Drin$ such that $\phi$ is standard, i.e.
$${\rm Ker} \exp_\phi =\widetilde{\pi} A.$$
Let $f\in \Sht$ be the shtuka function associated to $\phi.$
 \begin{theorem}\label{TheoremS3-2}
 Let $n\geq 1,$ $n\equiv 0\pmod{q^{d_\infty}-1}.$ Then, there exists $b \in B'\setminus\{0\}$ such that we have the following property in $\mathbb C_\infty:$
 $$\frac{\sum_I\frac{\sigma_I(b)}{\psi_{\phi} (I)^{n}}}{\widetilde{\pi}^{n}}\in H_A^\times.$$
 \end{theorem}
 
\begin{proof}

Write $n=q^k-s,$ $k\equiv 0\pmod{d_\infty}, s\equiv 1\pmod{q^{d_\infty}-1}.$

\medskip

\noindent 1) Observe that the map $u_.$ extends naturally into a map $u_.: \mathcal I(A)\rightarrow \mathbb H^\times,$ such that:
$$\forall x\in K^\times, \quad u_{xA}= \frac{\rho(x)}{\sgn(x)},$$
$$\forall I, J\in \mathcal I(A), \quad u_{IJ}= \sigma_I(u_J) u_I.$$
By Lemma \ref{LemmaS2-1},  we deduce that for all $l\geq 0,$ $\frac{\tau^l(u_I)}{u_I}$ has no zero and no pole at $\xi.$
 For $m\geq 1,$ $m\equiv 0\pmod{d_\infty},$  let $\chi_m: \mathcal I_A\rightarrow H_A^\times,$ such that:
 $$\forall I\in \mathcal I(A), \quad \chi_m(I)= \frac{\tau^m (u_I)}{u_I}\mid_{\xi}.$$
We observe that:
 $$\forall x\in K^\times, \quad \chi_m(xA)=1,$$
 $$ \forall I,J\in \mathcal I (A), \quad \chi_m (IJ)=\sigma_I(\chi_m(J))\chi_m(I).$$
 In particular, there exists $b_m\in B'\setminus\{0\}$ such that:
 $$\forall I\in \mathcal I(A), \quad \chi_m(I)=\frac{\sigma_I(b_m)}{b_m}.$$
 
\medskip

\noindent 2) By Theorem \ref{TheoremS3-1}, we have:
 $$\frac{\mathcal L_s(1) \prod_{j=1}^s \omega_j}{\widetilde{\pi}}\in \mathbb K_s.$$
We now apply $\tau^k$ to the above rationality result. We  get:
 $$\frac{\prod_{j=1}^s(f_j\cdots f_j^{(k-1)}\omega_j)\,  \tau^k(\mathcal L_s(1))}{\widetilde{\pi}^{q^k}}\in \mathbb K_s.$$
Let $j\in \{1, \ldots, s\}.$  Let $\mathbb H_{s,j}=H(\rho_k(K), k=1, \ldots, s, k\not =j).$  Let $\xi_j$ be the place of $\mathbb H_s/\mathbb H_{s,j}$ which  corresponds to the kernel of the homomorphism of $\mathbb H_{s,j}$-algebras: $\rho_j(A)[\mathbb H_{s,j}]\rightarrow \mathbb H_{s, j}, \rho_j(a)\mapsto a.$ By Proposition  \ref{PropositionS2-2}, there exists $x_j\in K(\rho_j(K))^{\times}$ such that  we have (recall that $e_1(\phi)\not =0$) :
 $$x_jf_j\cdots f_j^{(k-1)}\omega_j \mid_{\xi_j} \in \widetilde{\pi}H_A^\times.$$
Now:
 $$\tau^k(\mathcal L_s(1))=\sum_I \frac{\prod_{j=1}^s \rho_j(u_I)}{\psi_{\phi}(I)^{q^k}} \prod_{j=1}^s\frac{\tau^k(\rho_j(u_I))}{\rho_j(u_I)} .$$
Therefore, there exists $b\in B'\setminus\{0\}$ such that:
 $$\tau^k(\mathcal L_s(1))\mid_{\xi_1, \ldots, \xi_s}= \frac{1}{b} \prod_P(1-\frac{1}{\psi_{\phi}(P)^{q^k-s}}(P, H/K))^{-1}(b) \in K_\infty^\times.$$
The Theorem follows.
\end{proof}


\begin{thebibliography}{23}

\bibitem{AND2} G. Anderson, Log-Algebraicity of Twisted $A$-Harmonic Series and Special Values of $L$-series in Characteristic $p$, {\it Journal of Number Theory} {\bf 60} (1996), 165-209.
 \bibitem{AND} G. Anderson, Rank one elliptic $A$-modules and $A$-harmonic series, {\it Duke Mathematical Journal}
 {\bf 73} (1994), 491-542.
  \bibitem{AND&THA} G. Anderson, D. Thakur, Tensor Powers of the Carlitz Module and Zeta Values, {\it Annals of Mathematics} {\bf 132}(1990), 159-191.
 \bibitem{ANDTR1} B. Angl\`es, T. Ngo Dac, F. Tavares Ribeiro, Stark units in positive characteristic, arXiv: 1606.05502 (2016).
\bibitem{ANDTR2} B. Angl\`es, T. Ngo Dac, F. Tavares Ribeiro, Twisted Characteristic $p$ Zeta Functions,  {\it Journal of Number Theory} {\bf 168} (2016), 180-214.
\bibitem{ANDTR3} B. Angl\`es, T. Ngo Dac, F. Tavares Ribeiro, Exceptional Zeros of $L$-series and Bernoulli-Carlitz Numbers (with an appendix by B. Angl\`es, D. Goss, F. Pellarin, F. Tavares Ribeiro), arXiv: 1511.06209 (2015) .
\bibitem{ANG&PEL} B. Angl\`es, F. Pellarin, Universal Gauss-Thakur sums and $L$-series, {\it Inventiones mathematicae} {\bf 200} (2015), 653-669.
\bibitem{APT} B. Angl\`es, F. Pellarin, F. Tavares Ribeiro, with an appendix by F. Demeslay,  Arithmetic of positive characteristic $L$-series values in Tate algebras,  {\it Compositio Mathematica} {\bf 152} (2016), 1-61.
\bibitem{APTR} B. Angl\`es, F. Pellarin, F. Tavares Ribeiro, Anderson-Stark Units for $\mathbb F_q[\theta],$ to appear in {\it Transactions of the American Mathematical Society}, arXiv: 1501.06804 (2015).
\bibitem{AT} B. Angl\`es, L. Taelman, with an appendix by V. Bosser, Arithmetic of characteristic $p$ special $L$-values, {\it Proceedings of the London Mathematical Society} {\bf 110} (2015), 1000-1032.
\bibitem{ATR} B. Angl\`es, F. Tavares Ribeiro, Arithmetic of function fields units, to appear in {\it Mathematische Annalen}, DOI 10.1007/s00208-016-1405-2.
\bibitem{DEM} F. Demeslay, A class formula for $L$-series in positive characteristic, arXiv:1412.3704 (2014).
\bibitem{FAN0} J. Fang, Equivariant trace formula mod $p,$ {\it Comptes Rendus Math\'ematique} {\bf 354} (2016), 335-338.
\bibitem{FAN1} J. Fang, Special $L$-values of abelian $t$-modules, {\it Journal of Number Theory} {\bf 147} (2015), 300-325.
\bibitem{FAN2} J. Fang, Equivariant Special $L$-values of abelian $t$-modules, arXiv: 1503.07243 (2015).
\bibitem{GOS} D. Goss, Basic Structures of Function Field Arithmetic, Springer, 1996.
\bibitem{GRE&PAP} N. Green, M. Papanikolas, Special $L$-values and shtuka functions for Drinfeld modules on elliptic curves,  arXiv:1607.04211 (2016).
\bibitem{PAP} M. Papanikolas, Log-Algebraicity on Tensor Powers of the Carlitz Module and Special Values of Goss $L$-Functions, work in progress.
\bibitem{PEL} F. Pellarin, Values of certain $L$-series in positive characteristicc, {\it Annals of Mathematics} {\bf 176} (2012), 2055-2093.
\bibitem{TAE1} L. Taelman, A Dirichlet unit theorem for Drinfeld modules, {\it Mathematische Annalen} {\bf 348} (2010), 899-907.
\bibitem{TAE2} L. Taelman, Special $L$-values of Drinfeld modules, {\it Annals of Mathematics} {\bf 175} (2012), 369-391.
\bibitem{THA2} D. Thakur, Shtukas and Jacobi sums, {\it Inventiones mathematicae} {\bf 111} (1993), 557-570.
\bibitem{THA3} D. Thakur, Gauss sums for $\mathbb F_q[T],$ {\it Inventiones mathematicae} {\bf 94}(1988), 105-112.

\end{thebibliography}
 \end{document}